\documentclass[a4paper,twoside]{amsart}
\usepackage{amssymb} 
\usepackage[all]{xy} 
\usepackage{color} 
\usepackage{lastpage} 
\usepackage[pdfstartview=FitH]{hyperref} 

\newtheorem{thm}{Theorem}[section]

\newtheorem{lem}[thm]{Lemma}

\newtheorem{prop}[thm]{Proposition}

\theoremstyle{definition}
\newtheorem{defn}[thm]{Definition}
\theoremstyle{remark}
\newtheorem{rem}[thm]{Remark}

\numberwithin{equation}{section}

\newcommand*{\cat}[1]{\mathbf{#1}} 
\newcommand*{\mto}{\rightarrow} 
\newcommand*{\set}[1]{\left\{#1\right\}} 
\newcommand*{\comp}{\circ} 
\newcommand*{\Int}{\mathbb{Z}} 
\newcommand*{\Rat}{\mathbb{Q}} 
\newcommand*{\FF}{\mathbb{F}} 
\newcommand*{\cmplx}[1]{{#1}^\bullet} 
\newcommand*{\tensor}{\otimes} 
\newcommand*{\isomorph}{\cong} 
\newcommand*{\op}{\mathrm{op}} 
\newcommand*{\cont}{\mathrm{cont}} 
\newcommand*{\sep}{\mathrm{sep}} 
\newcommand*{\algc}[1]{\overline{#1}} 
\newcommand*{\sheaf}[1]{\mathcal{#1}} 
\newcommand*{\talg}[2][T]{{#2}\langle #1\rangle}
\newcommand{\openideals}{\mathfrak{I}} 
\newcommand{\Sect}{\Gamma} 
\newcommand{\Sectc}{\Gamma_{\mathrm{c}}} 
\newcommand{\ab}{\mathrm{ab}} 
\newcommand{\loc}{\mathrm{loc}} 
\newcommand{\rel}{\mathrm{rel}} 
\newcommand{\ncL}{\mathcal{L}} 
\newcommand{\ValR}{\mathcal{O}} 

\newcommand{\ringtransf}{\Psi} 
\newcommand{\id}{\mathrm{id}} 
\newcommand{\Frob}{\mathfrak{F}} 

\DeclareMathOperator{\HF}{H} 
\DeclareMathOperator{\Jac}{Jac} 
\DeclareMathOperator{\im}{im} 
\DeclareMathOperator{\Spec}{Spec} 
\DeclareMathOperator{\RDer}{R} 
\DeclareMathOperator{\Gal}{Gal} 
\DeclareMathOperator{\cd}{cd} 
\DeclareMathOperator{\KTh}{K} 
\DeclareMathOperator{\SK}{SK} 
\DeclareMathOperator{\Det}{Det} 
\DeclareMathOperator{\Res}{Res} 

\newdir{ >}{{}*!/-5pt/@{>}} 
\newdir{O}{{}*-\cir<2.5pt>{}} 
\xyoption{2cell} 
\UseTwocells %
\xyoption{poly}

\allowdisplaybreaks[1]

\begin{document}


\title[Unit $L$-functions]{Unit $L$-Functions for \'etale sheaves of modules over noncommutative rings}%
\author{Malte Witte}%

\address{Malte Witte\newline Ruprecht-Karls-Universit\"at Heidelberg\newline
Mathematisches Institut\newline
Im Neuenheimer Feld 288\newline
D-69120 Heidelberg }%
\email{witte@mathi.uni-heidelberg.de}


\date{\today}%

\begin{abstract}
Let $s\colon X\mto \Spec \FF$ be a separated scheme of finite type over a finite field $\FF$ of characteristic $p$, let $\Lambda$ be a not necessarily commutative $\Int_p$-algebra with finitely many elements, and let $\cmplx{\sheaf{F}}$ be a perfect complex of $\Lambda$-sheaves on the \'etale site of $X$. We show that the ratio $L(\cmplx{\sheaf{F}},T)/L(\RDer s_!\cmplx{\sheaf{F}},T)$, which is a priori an element of $\KTh_1(\Lambda[[T]])$, has a canonical preimage in $\KTh_1(\Lambda[T])$. We use this to prove a version of the noncommmutative Iwasawa main conjecture for $p$-adic Lie coverings of $X$.
\end{abstract}

\maketitle

Let $p$ be a prime number and $\FF$ the finite field with $q=p^v$ elements, and $s\colon X\mto \Spec \FF$ a separated finite type $\FF$-scheme. Let further $\Lambda$ be an adic $\Int_p$-algebra, i.\,e.\ $\Lambda$ is compact for the topology defined by the powers of its Jacobson radical $\Jac(\Lambda)$. For any perfect complex of $\Lambda$-sheafs $\cmplx{\sheaf{F}}$ on the \'etale site of $X$ we have defined in \cite{Witte:NoncommutativeLFunctions} an $L$-function $L(\cmplx{\sheaf{F}},T)$ attached to $\cmplx{\sheaf{F}}$. This is an element in the first $K$-group $\KTh_1(\Lambda[[T]])$ of the power series ring $\Lambda[[T]]$ in the formal variable $T$ commuting with the elements of $\Lambda$. The total higher direct image $\RDer f_!\cmplx{\sheaf{F}}$ is again a perfect complex of $\Lambda$-sheaves and so we may form the $L$-function $L(\RDer f_!\cmplx{\sheaf{F}},T)$, which is also an element of $\KTh_1(\Lambda[[T]])$. Different from the situation where $\Lambda$ is an adic $\Int_\ell$-algebra with $\ell\neq p$ dicussed in \cite{Witte:NoncommutativeLFunctions}, the ratio $L(\cmplx{\sheaf{F}},T)/L(\RDer f_!\cmplx{\sheaf{F}},T)$ does not need to be $1$ in $\KTh_1(\Lambda[[T]])$.

Let
$$
\talg{\Lambda}=\varprojlim_k \Lambda/\Jac(\Lambda)^k[T]
$$
denote the $\Jac(\Lambda)$-adic completion of the polynomial ring $\Lambda[T]$ and let
$$
\widehat{\KTh}_1(\talg{\Lambda})=\varprojlim_k \KTh_1(\Lambda/\Jac(\Lambda)^k[T])
$$
denote its first completed $K$-group. If $\Lambda$ is commutative, then
$$
\widehat{\KTh}_1(\talg{\Lambda})=\KTh_1(\talg{\Lambda})=\talg{\Lambda}^\times
$$
is a subgroup of $\KTh_1(\Lambda[[T]])=\Lambda[[T]]^\times$ and Emerton and Kisin \cite{EmertonKisin} show that $L(\cmplx{\sheaf{F}},T)/L(\RDer f_!\cmplx{\sheaf{F}},T)\in\talg{\Lambda}^\times$. If $\Lambda$ is not commutative, the canonical homomorphism
$$
\widehat{\KTh}_1(\talg{\Lambda})\mto \KTh_1(\Lambda[[T]])
$$
is no longer injective in general. Nevertheless, we shall prove:

\begin{thm}[see Thm.~\ref{thm:construction of Q}]
There exists a unique way to associate to each separated $\FF$-scheme $s\colon X\mto \Spec \FF$ of finite type, each adic $\Int_p$-algebra $\Lambda$, and each perfect complex of $\Lambda$-sheaves $\cmplx{\sheaf{F}}$ on $X$ an element $Q(\cmplx{\sheaf{F}},T)\in \widehat{\KTh}_1(\talg{\Lambda})$ such that
\begin{enumerate}
\item the image of $Q(\cmplx{\sheaf{F}},T)$ in $\KTh_1(\Lambda[[T]])$ is the ratio $L(\cmplx{\sheaf{F}},T)/L(\RDer s_!\cmplx{\sheaf{F}},T)$,
\item $Q(\cmplx{\sheaf{F}},T)$ is multiplicative on exact sequences of perfect complexes and depends only on the quasi-isomorphism class of $\cmplx{\sheaf{F}}$,
\item $Q(\cmplx{\sheaf{F}},T)$ is compatible with changes of the ring $\Lambda$.
\end{enumerate}
\end{thm}

Aside from the result of Emerton and Kisin, a central ingredient for the proof is the recent work of Chinburg, Pappas, and Taylor \cite{CPT:K1ofPadicGroupRingI}, \cite{CPT:K1ofPadicGroupRingII} on the first $\KTh$-group of $p$-adic group rings. In fact, the main strategy of the proof is to reduce the assertion first to the case $\Lambda=\Int_p[G]$ for a finite group $G$ and then use the results of Chinburg, Pappas, and Taylor to reduce it further to the case already treated by Emerton and Kisin. In particular, we use almost exclusively methods from representation theory, whereas the result of Emerton and Kisin itself may be considered as the geometric input.

As an application, we deduce the following version of a noncommutative Iwasawa main conjecture  for varieties over finite fields. Assume for the moment that $X$ is geometrically connected and let $G$ be a factor group of the fundamental group of $X$ such that $G\isomorph H\rtimes \varGamma$ where $H$ is a compact $p$-adic Lie group and
$$
\varGamma=\Gal(\FF_{q^{p^{\infty}}}/\FF)\isomorph\Int_{p}.
$$
We write
$$
\Int_{p}[[G]]=\varprojlim \Int_{p}[G/U]
$$
for the Iwasawa algebra of $G$. Let
$$
S=\{f\in \Int_{p}[[G]]\colon \text{$\Int_{p}[[G]]/\Int_{p}[[G]]f$ is finitely generated as $\Int_{p}[[H]]$-module}\}
$$
denote Venjakob's canonical Ore set and write $\Int_{p}[[G]]_S$ for the localisation of $\Int_{p}[[G]]$ at $S$. We turn $\Int_{p}[[G]]$ into a smooth $\Int_{p}[[G]]$-sheaf $\sheaf{M}(G)$ on $X$ by letting the fundamental group of $X$ act contragrediently on $\Int_{p}[[G]]$. Let $\RDer\Sectc(X,\sheaf{F})$  denote the \'etale cohomology with proper support of a flat constructible $\Int_{p}$-sheaf $\sheaf{F}$ on $X$.

For every continuous $\Int_{p}$-representation $\rho$ of $G$, there exists a homomorphism
$$
\rho\colon \KTh_1(\Int_{p}[[G]]_S)\mto Q(\Int_{p}[[\varGamma]])^{\times}
$$
into the units $Q(\Int_{p}[[\varGamma]])^{\times}$ of the field of fractions of $\Int_{p}[[\varGamma]]$. It is induced by sending $g\in G$ to $\det ([g]\rho(g)^{-1})$, with $[g]$ denoting the image of $g$ in $\varGamma$. On the other hand,  $\rho$ gives rise to a flat and smooth $\Int_{p}$-sheaf $\sheaf{M}(\rho)$ on $X$.

\begin{thm}\label{thm:special case NCIMC}\
\begin{enumerate}
\item $\RDer\Sectc(X,\sheaf{M}(G)\tensor_{\Int_p}\sheaf{F})$ is a perfect complex of $\Int_{p}[[G]]$-modules whose cohomology groups are $S$-torsion. In particular, it gives rise to a class
$$
[\RDer\Sectc(X,\sheaf{M}(G)\tensor_{\Int_p}\sheaf{F})]^{-1}
$$
in the relative $\KTh$-group $\KTh_0(\Int_{p}[[G]],\Int_{p}[[G]]_S)$.
\item There exists an element $\widetilde{\ncL}_G(X/\FF,\sheaf{F})\in\KTh_1(\Int_{p}[[G]]_S)$ with the following properties:
\begin{enumerate}
\item \emph{(Characteristic element)} The image of $\widetilde{\ncL}_G(X/\FF,\sheaf{M}(\rho)\tensor_{\Int_{p}}\sheaf{F})$ under the boundary homomorphism $d\colon\KTh_1(\Int_{p}[[G]]_S)\mto \KTh_0(\Int_{p}[[G]],\Int_{p}[[G]]_S)$ is
$$
[\RDer\Sectc(X,\sheaf{M}(G)\tensor_{\Int_p}\sheaf{F})]^{-1}.
$$
\item \emph{(Interpolation with respect to all continuous representations)} Assume that $\rho$ is a continuous $\Int_{p}$-representation of $G$. We let $\gamma$ denote the image of the geometric Frobenius in $\varGamma$. Then $$\rho(\widetilde{\ncL}_G(X/\FF,\sheaf{F}))=L(\sheaf{M}(\rho)\tensor_{\Int_p}\sheaf{F},\gamma^{-1})$$
    in $Q(\Int_{p}[[\varGamma]])^{\times}$.
\end{enumerate}
\end{enumerate}
\end{thm}

This enhances the main result of \cite{Burns:MCinGIwTh+RelConj}, which also asserts the existence of $\widetilde{\ncL}_G(X/\FF,\sheaf{F})$, but requires only that it satisfies the interpolation property with respect to finite representations. In fact, we will prove in Section~\ref{sec:MC} an even more general version of this theorem in the style of \cite{Witte:MCVarFF}, replacing $\Int_{p}$ by arbitrary adic $\Int_p$-algebras and allowing schemes and coverings which are not necessarily connected. 

We refer to \cite{Burns:MCinGIwTh+RelConj} and \cite{Witte:NCIMCFuncField} for applications of Theorem~\ref{thm:special case NCIMC}.

\section{Preliminaries on completed \texorpdfstring{$\KTh$}{K}-Theory}

For any topological ring $R$ (associative and with unity) we let $\openideals_R$ denote the lattice of all two-sided open ideals of $R$. For any $n\geq 0$ we call
$$
\widehat{\KTh}_n(R)=\varprojlim_{I\in\openideals_R}\KTh_n(R/I)
$$
the $n$-th completed $\KTh$-group of $R$. The group $\widehat{\KTh}_n(R)$ becomes a topological group by equipping each $\KTh_n(R/I)$ with the discrete topology.

Recall that we call $R$ an adic ring if $R$ is compact and the Jacobson radical $\Jac(R)$ is open and finitely generated, or equivalently,
$$
R=\varprojlim_k R/\Jac(R)^k
$$
with $R/\Jac(R)^k$ a finite ring. Fukaya and Kato showed that for adic rings the canonical homomorphism $\KTh_1(R)\mto\widehat{\KTh}_1(R)$ is an isomorphism \cite[Prop.~1.5.1]{FK:CNCIT}. The same is true if $R=A[G]$ with $G$ a finite group and $A$ a commutative, $p$-adically complete, Noetherian integral domain with fraction field of characteristic $0$ \cite[Thm.~1.2]{CPT:K1ofPadicGroupRingII}. We can add the following rings to the list.

\begin{prop}
Let $R$ be a commutative topological ring such that
\begin{enumerate}
  \item the topology of $R$ is an ideal topology, i.\,e.\ $\openideals_R$ is a basis of open neighbourhoods of $0$,
  \item $$R=\varprojlim_{I\in\openideals_R}R/I,$$
  \item the Jacobson radical $\Jac(R)$ is open.
\end{enumerate}
Then $\KTh_1(R)\mto\widehat{\KTh}_1(R)$ is an isomorphism.
\end{prop}
\begin{proof}
For any commutative ring $A$ we have an exact sequence
$$
1\mto 1+\Jac(A)\mto \KTh_1(A)\mto \KTh_1(A/\Jac(A))\mto 1.
$$
Since $R/\Jac(R)$ carries the discrete topology we have
$$
\KTh_1(R/\Jac(R))=\widehat{\KTh}_1(R/\Jac(R)).
$$
Because of the completeness of $R$ and because $\Jac(R)$ is open we have
$$
1+\Jac(R)=\varprojlim_{I\in\openideals_R} 1+\Jac(R)/\Jac(R)\cap I.
$$
The claim now follows from the snake lemma.
\end{proof}

Let $\Lambda$ be an adic ring. We are mainly interested in the topological rings
$$
\talg{\Lambda}=\varprojlim_{I\in\openideals_\Lambda}\Lambda/I[T]
$$
where we give the polynomial ring $\Lambda/I[T]$ the discrete topology. As every open two-sided ideal $I$ of $\Lambda$ is finitely generated as left or right ideal, we note that $I\talg{\Lambda}=\ker(\talg{\Lambda}\mto \Lambda/I[T])$ and $\Jac(\talg{\Lambda})=\Jac(\Lambda)\talg{\Lambda}$. In particular, the open ideals of $\talg{\Lambda}$ are again finitely generated.

We suspect that for all these rings $\KTh_1(\talg{\Lambda})$ and $\widehat{\KTh}_1(\talg{\Lambda})$ agree, but we were not able to prove this in general. By the above results they do agree if either $\Lambda$ is commutative or $\Lambda=\Int_p[G]$ for a finite group $G$ and this is all we need for our purposes. The following result is therefore only for the reader's edification.

\begin{prop}\label{prop:units surject onto K1}
Let $\Lambda$ be an adic ring. Then the homomorphisms
$$
\talg{\Lambda}^\times\mto \KTh_1(\talg{\Lambda})\mto \widehat{\KTh}_1(\talg{\Lambda})
$$
are surjective.
\end{prop}
\begin{proof}
We note that $\Lambda/\Jac(\Lambda)$ is a finite product of finite-dimensional matrix rings over finite fields. In particular,
$\KTh_n(\Lambda/\Jac(\Lambda)[T])=\KTh_n(\Lambda/\Jac(\Lambda))$ for all $n$ by a celebrated result of Quillen. Since $\Lambda/\Jac(\Lambda)$ is semi-simple, the homomorphism $\Lambda/\Jac(\Lambda)^\times\mto\KTh_1(\Lambda/\Jac(\Lambda))$ is surjective. Moreover, $\KTh_2(\Lambda/\Jac(\Lambda))=0$ since this is true for all finite fields.

By a result of Vaserstein (see \cite[Thm.~1.5]{Oliver:WhiteheadGroups}) we have for any ring $R$ and any two-sided ideal $I\subset \Jac(R)$ an exact sequence
$$
1\mto V(R,I)\mto 1+I\mto \KTh_1(R,I)\mto 1
$$
with $V(R,I)$ the subgroup of $1+I$ generated by the elements $(1+ri)(1+ir)^{-1}$ with $r\in R$ and $i\in I$. Choosing $R=\talg{\Lambda}$ and $I=\Jac(\Lambda)\talg{\Lambda}$ we conclude that $\talg{\Lambda}^\times\mto\KTh_1(\talg{\Lambda})$ is surjective.

Since the homomorphisms
$$
V(\talg{\Lambda},\Jac(\talg{\Lambda}))\mto V(\Lambda/I[T],\Jac(\Lambda/I[T]))
$$
are surjective for any open two-sided ideal $I\subset \Jac(\Lambda)$ of $\Lambda$, we conclude that
$$
\RDer^1\varprojlim_k V(\Lambda/\Jac(\Lambda)^k[T],\Jac(\Lambda)\Lambda/\Jac(\Lambda)^k[T])=0.
$$
Hence,
$$
\KTh_1(\talg{\Lambda},\Jac(\talg{\Lambda}))\mto \varprojlim_k \KTh_1(\Lambda/\Jac(\Lambda)^k[T],\Jac(\Lambda)\Lambda/\Jac(\Lambda)^k[T])
$$
is surjective. Passing to the limit over the exact sequences
\begin{multline*}
1\mto\KTh_1(\Lambda/\Jac(\Lambda)^k[T],\Jac(\Lambda)\Lambda/\Jac(\Lambda)^k[T])\\
\mto\KTh_1(\Lambda/\Jac(\Lambda)^k[T])\mto\KTh_1(\Lambda/\Jac(\Lambda)[T])\mto 1
\end{multline*}
and comparing it to the corresponding sequence for $\talg{\Lambda}$ we conclude that
$$
\KTh_1(\talg{\Lambda})\mto\widehat{\KTh}_1(\talg{\Lambda})
$$
is surjective.
\end{proof}

If $\Lambda$ is an adic ring and $P$ a finitely generated, projective left $\Lambda$-module, we set
$$
P[T]=\Lambda[T]\tensor_{\Lambda}P,\quad \talg{P}=\talg{\Lambda}\tensor_{\Lambda}P, \quad P[[T]]=\Lambda[[T]]\tensor_{\Lambda}P
$$
Note that
$$
\talg{P}=\varprojlim_k P/\Jac(\Lambda)^kP[T],\qquad P[[T]]=\varprojlim_{k} P/\Jac(\Lambda)^k P[[T]],
$$
and that $\Lambda[[T]]$ is again an adic ring.

If $\Lambda'$ is another adic ring acting on $P$ from the right such that $P$ becomes a $\Lambda$-$\Lambda'$-bimodule, then $P/\Jac(\Lambda)P[T]$ is annihilated by some power $\Jac(\Lambda')^m(k)$ of the Jacobson radical of $\Lambda'$. This shows that $\talg{P}$ is a $\talg{\Lambda}$-$\talg{\Lambda'}$-bimodule and therefore induces homomorphisms
$$
\ringtransf_{\talg{P}}\colon \KTh_n(\talg{\Lambda'})\mto \KTh_n(\talg{\Lambda}).
$$
At the same time, it shows that the system $(P/\Jac(\Lambda)^k P[T])_{k\geq 1}$ of $\Lambda/\Jac(\Lambda)^k$-$\Lambda'/\Jac(\Lambda')^{m(k)}$-bimodules induces homomorphisms
$$
\ringtransf_{\talg{P}}\colon \widehat{\KTh}_n(\talg{\Lambda'})\mto\widehat{\KTh}_n(\talg{\Lambda}),
$$
which are compatible with the above homomorphisms. The construction of $\ringtransf_{\talg{P}}$ extends in the obvious manner to complexes $\cmplx{P}$ of $\Lambda$-$\Lambda'$-bimodules which are strictly perfect as complexes of $\Lambda$-modules. By a similar reasoning we also obtain change-of-ring homomorphisms
$$
\ringtransf_{\cmplx{P[[T]]}}\colon \KTh_n(\Lambda'[[T]])\mto \KTh_n(\Lambda[[T]]),
$$
as well as the corresponding versions for the completed $\KTh$-theory.

\section{On \texorpdfstring{$\KTh_1(\talg{\Int_p[G]})$}{K1(Zp[G]<T>)}}

In this section, we use the results of Chinburg, Pappas, and Taylor \cite{CPT:K1ofPadicGroupRingI}, \cite{CPT:K1ofPadicGroupRingII} to analyse $\KTh_1(\talg{\Int_p[G]})$ for a finite group $G$.

For a noetherian integral domain of finite Krull dimension $R$ with field of fractions $Q(R)$ of characteristic $0$ we set
\begin{align*}
\SK_1(R[G])&=\ker(\KTh_1(R[G])\mto \KTh_1(\algc{Q(R)}[G])),\\
\Det(R[G]^\times)&=\im(R[G]^\times\mto\KTh_1(\talg{\ValR_K[G]}))/\SK_1(\talg{\ValR_K[G]}).
\end{align*}
Here, $\algc{Q(R)}$ denotes a fixed algebraic closure of $Q(R)$. For any subgroup $U$ of $R[G]^\times$ we write $\Det(U)$ for its image in $\Det(R[G]^\times)$. We are mainly interested in the cases $R=\talg{\ValR_K}$ or $R=\ValR_K[[T]]$ with $\ValR_K$ the valuation ring of a finite extension $K/\Rat_p$.

Assume that $L/K$ is a finite extension such that $L$ is a splitting field for $G$, i.\,e.\ $L[G]$ is a finite product of matrix algebras over $L$. Let $\mathcal{M}$ be a maximal $\Int_p$-order inside $L[G]$ containing $\ValR_L[G]$. Then $\mathcal{M}$ is a finite product of matrix algebras over $\ValR_L$ such that
$$
\KTh_1(\talg{\mathcal{M}})\mto \KTh_1(Q(\talg{\ValR_K})[G])
$$
is injective. In particular, we have
$$
\SK_1(\talg{\ValR_K[G]})=\ker(\KTh_1(\talg{\ValR_K[G]})\mto \KTh_1(\talg{\mathcal{M}})).
$$
The same reasoning also applies to $\SK_1(\ValR_K[G][[T]])$. Moreover, note that the group $\KTh_1(\talg{\mathcal{M}})$ injects into $\KTh_1(\mathcal{M}[[T]])$, such that we also have an injection $\Det(\talg{\ValR_K[G]}^\times)\subset\Det(\ValR_K[G][[T]]^\times)$.

\begin{lem}\label{lem:SK1 for power series rings}
For any finite group $G$ and any finite field extension $K/\Rat_p$, the inclusion $\ValR_K[G]\mto\ValR_K[G][[T]]$ induces an isomorphism
$$
\SK_1(\ValR_K[G])\isomorph\SK_1(\ValR_K[G][[T]]).
$$
In particular,
$$
\KTh_1(\ValR_K[G][[T]],T\ValR_K[G][[T]])=\Det(1+T\ValR_K[G][[T]]).
$$
\end{lem}
\begin{proof}
The first equality is proved in \cite[Prop.~5.4]{Witte:NoncommutativeLFunctions}. Since $T\in\Jac(\ValR_K[G][[T]])$, the map
$$
1+T\ValR_K[G][[T]]\mto\KTh_1(\ValR_K[G][[T]],T\ValR_K[G][[T]])
$$
is surjective by the result of Vaserstein \cite[Thm.~1.5]{Oliver:WhiteheadGroups} that we already used in the proof of Prop~\ref{prop:units surject onto K1}. The decomposition
$$
\KTh_1(\ValR_K[G][[T]])=\KTh_1(\ValR_K[G])\times\KTh_1(\ValR_K[G][[T]],T\ValR_K[G][[T]])
$$
induced by the inclusion $\ValR_K[G]\mto \ValR_K[G][[T]]$ and the evaluation $T\mapsto 0$ shows that
$$
\SK_1(\ValR_K[G][[T]])\cap \KTh_1(\ValR_K[G][[T]],T\ValR_K[G][[T]])=1.
$$
\end{proof}

\begin{lem}\label{lem:p torsion criterion}
Let $L/K/\Rat_p$ be finite extensions such that $L$ is a splitting field for $G$ and let $\mathcal{M}$ be a maximal $\Int_p$-order inside $L[G]$ containing $\ValR_L[G]$. Assume that $f$ is in the intersection of $\KTh_1(\talg{\mathcal{M}},\Jac(\talg{\mathcal{M}}))$ and $\Det(1+\Jac(\ValR_K[G])\ValR_K[G][[T]])$ inside $\KTh_1(\mathcal{M}[[T]])$. Then there exists $n\geq 0$ such that $f^{p^n}\in\Det(\talg{\ValR_K[G]})$.
\end{lem}
\begin{proof}
Let $p^k$ be the order of a $p$-Sylow subgroup of $G$. Since $\mathcal{M}/p^{k+2}\mathcal{M}$ is a finite ring, some power of $\Jac(\mathcal{M})$ is contained in $p^{k+2}\mathcal{M}$. Now $p^k\mathcal{M}\subset\ValR_L[G]$
\cite[Thm.~1.4]{Oliver:WhiteheadGroups}. For large $n$ we thence have
$$
f^{p^n}\in\Det(1+p^2\talg{\ValR_L[G]})\cap\Det(1+p^2\ValR_K[G][[T]]).
$$
By \cite[Prop.~2.4]{CPT:K1ofPadicGroupRingI} the $p$-adic logarithm induces $R$-linear isomorphisms
$$
v\colon \Det(1+p^2R[G])\mto p^2R[C_G]
$$
where $C_G$ is the set of conjugacy classes of $G$ and $R$ is equal to either $\talg{\ValR_K}$ or $\ValR_K[[T]]$ for any $K$. In particular,
$$
v(f^{p^n})\in p^2\talg{\ValR_L[C_G]}\cap p^2\ValR_K[[T]][C_G]=p^2\talg{\ValR_K}[C_G].
$$
Hence, $f^{p^n}\in \Det(1+p^2\talg{\ValR_K[G]})$.
\end{proof}

For any finite group $G$ we let $[G,G]$ denote the commutator subgroup of $G$ and set $G^\ab=G/[G,G]$.

\begin{lem}\label{lem:torsionfreeness p group case}
Let $G$ be a finite $p$-group and $K/\Rat_p$ a finite unramified extension. Let further $A$ denote the kernel of $\ValR_K[G]\mto \ValR_K[G^\ab]$. The abelian group
$$
\Det(1+A\ValR_K[G][[T]])/\Det(1+A\talg{\ValR_K[G]})
$$
is torsionfree.
\end{lem}
\begin{proof}
Let $R$ be equal to either $\talg{\ValR_K}$ or $\ValR_K[[T]]$. Write $C_G$ for set of the conjugacy classes of $G$ and let $\phi(AR[G])$ denote the kernel of the natural $R$-linear map $R[C_G]\mto R[G^\ab]$. Note that $\phi(AR[G])$ is finitely generated and free as an $R$-module.
For any choice of Frobenius lifts compatible with the inclusion $\talg{\ValR_K[G]}\subset\ValR_K[[T]]$ we obtain a commutative diagram
$$
\xymatrix{\Det(1+A\talg{\ValR_K[G]})\ar[r]\ar[d]^{\isomorph}_{\nu}&\Det(1+A\ValR_K[G][[T]])\ar[d]^\isomorph_{\nu}\\
          p\phi(A\talg{\ValR_K[G]})\ar[r]&p\phi(A\ValR_K[G][[T]])}
$$
with the horizontal arrows induced by the natural inclusion and the vertical isomorphisms induced by the integral group logarithm \cite[Thm.~3.16]{CPT:K1ofPadicGroupRingI}. Since $\ValR_K[[T]]/\talg{\ValR_K}$ is a torsionfree abelian group, the same is true for the group $p\phi(A\ValR_K[G][[T]])/p\phi(A\talg{\ValR_K[G]})$.
\end{proof}

Recall that a semi-direct product $G=\Int/s\Int\rtimes P$ with $s$ prime to $p$ is called $p$-$\Rat_p$-elementary if $P$ is a $p$-group and the image of $t\colon P\mto(\Int/s\Int)^\times$ given by the action of $P$ on $\Int/s\Int$ lies in $\Gal(\Rat_p(\zeta_s)/\Rat_p)\subset(\Int/s\Int)^\times$. For any divisor $m$ of $s$, and $R=\Int_p$, $R=\talg{\Int_p}$, or $R=\Int_p[[T]]$ we set
$$
R[m]=\Int[\zeta_m]\tensor_{\Int}R
$$
and let $R[m][P;t]$ denote the twisted group ring for $t$, i.\,e.\ $\sigma r=t(\sigma)(r)\sigma$ for elements $r\in R[m]$, $\sigma\in P$. Set
$$
H_m=\ker(t\colon P\mto \Gal(\Rat_p(\zeta_m)/\Rat_p)),\qquad B_m=P/H_m.
$$
We may write
$$
R[G]=\prod_{m\mid s}\Int_p[m][P;t]
$$
\cite[Prop.~11.6]{Oliver:WhiteheadGroups}. We then see that $R[G]$ is a finitely generated, projective module over the subring
$$
\prod_{m\mid s}R[m][H_m]\subset\prod_{m\mid s}R[m][P;t].
$$
We let
$$
r\colon \Det(R[G]^\times)\mto \prod_{m\mid s}\Det(R[m][H_m]^\times)
$$
denote the corresponding restriction map. We further set
\begin{align*}
A&=\ker(\Int_p[G]\mto \prod_{m\mid s}\Int_p[m][P/[H_m,H_m];t]),\\
A_m&=\ker(\Int_p[m][H_m]\mto \Int_p[m][H_m^\ab]),
\end{align*}
and let $b_m$ denote the order of $B_m$.

\begin{lem}\label{lem:p elementary case}
With the notation as above,
\begin{enumerate}
\item $R[m][P/[H_m,H_m];t]$ is isomorphic to the ring of $b_m\times b_m$ matrices over its centre $R[m][H_m^\ab]^{B_m}$,
\item $r$ induces an isomorphism
$$
r\colon\Det(1+AR[G])\mto \prod_{m\mid s} \Det(1+A_mR[m][H_m])^{B_m}.
$$
\item $\Det(1+A\Int_p[G][[T]])/\Det(1+A\talg{\Int_p[G]})$ is a torsionfree abelian group.
\end{enumerate}
\end{lem}
\begin{proof}
Assertion $(1)$ is a theorem of Wall \cite[Thm.~8.3]{Wall:NormsOfUnits}. Assertion $(2)$ follows from \cite[Thm~6.2 and Diagramm~(6.7)]{CPT:K1ofPadicGroupRingI}. Assertion $(3)$ is then a consequence of Lemma~\ref{lem:torsionfreeness p group case} and $(2)$.
\end{proof}

We now return to the case that $G$ is an arbitrary finite group. For $R=\Int_p$, $R=\talg{\Int_p}$, and $R=\Int_p[[T]]$ and any subgroup $H\subset G$ we write $\Res_H^G$ for the change-of-ring homomorphism $\KTh_1(R[G])\mto \KTh_1(R[H])$ induced by the $\Int_p[H]$-$\Int_p[G]$-bimodule $\Int_p[G]$.

\begin{lem}\label{lem:induction}
Let $f\in \KTh_1(\Int_p[G][[T]])$  such that
\begin{enumerate}
\item for any $p$-$\Rat_p$-elementary group $H$, $\Res_H^G f$ is in the image of the homomorphism $\KTh_1(\talg{\Int_p[H]})\mto \KTh_1(\Int_p[H][[T]])$,
\item $f^{p^n}$ is in the image of $\KTh_1(\talg{\Int_p[G]})\mto \KTh_1(\Int_p[G][[T]])$ for some $n\geq 0$.
\end{enumerate}
Then $f$ is in the image of $\KTh_1(\talg{\Int_p[G]})\mto \KTh_1(\Int_p[G][[T]])$.
\end{lem}
\begin{proof}
The homomorphisms $\KTh_1(\talg{\Int_p[G]})\mto\KTh_1(\Int_p[G][[T]])$ for each finite group $G$ constitute a homomorphism of Green modules over the Green ring $G\mapsto \operatorname{G}_0(\Int_p[G])$ \cite[\S~4.c]{CPT:K1ofPadicGroupRingII}. Hence, \cite[Thm.~4.3]{CPT:K1ofPadicGroupRingII} implies that $f^\ell$ is in the image of $\KTh_1(\talg{\Int_p[G]})\mto \KTh_1(\Int_p[G][[T]])$ for some integer $\ell$ prime to $p$. Because of $(2)$ we may choose $\ell=1$.
\end{proof}

Finally, we need the following vanishing result for $\SK_1$, which is a variant of \cite[Prop.~2.3.7]{FK:CNCIT}.

\begin{prop}\label{prop:vanishing of SK}
Let $K/\Rat_p$ be unramified and $R=\talg{\ValR_K}$ or $R=\talg{\ValR_K[[T']]}$ for some indeterminate $T'$. (More generally, $R$ can be any ring satisfying the standing hypothesis of \cite{CPT:K1ofPadicGroupRingII}.) Let further
$\mathcal{G}$ be a profinite group with cohomological $p$-dimension $\cd_p \mathcal{G}\leq 1$. For any open normal subgroup $U\subset \mathcal{G}$ there exists an open subgroup $V\subset U$ normal in $\mathcal{G}$ such that the natural homomorphism
$$
\SK_1(R[\mathcal{G}/V])\mto \SK_1(R[\mathcal{G}/U])
$$
is the zero map.
\end{prop}
\begin{proof}
For any finite group $G$, let $G_r$ denote the set of $p$-regular elements, i.\,e. those elements in $G$ of order prime to $p$. The group $G$ acts on $G_r$ via conjugation. Write $\Int_p[G_r]$ for the free $\Int_p$-module generated by $G_r$. By \cite[Thm.~1.7]{CPT:K1ofPadicGroupRingII} there exists a natural surjection
$$
R\tensor_{\Int_p}\HF_2(G,\Int_p[G_r])\mto \SK_1(R[G]).
$$
Since $\HF_2(G,\Int_p[G_r])$ is finite for all finite groups $G$, it suffices to show that
$$
\varprojlim_{U\subset \mathcal{G}}\HF_2(\mathcal{G}/U,\Int_p[(\mathcal{G}/U)_r])=0,
$$
where the limit extends over all open normal subgroups of $\mathcal{G}$. After taking the Pontryagin dual we deduce the latter from
$$
\varinjlim_{U\subset \mathcal{G}}\HF^2(\mathcal{G},\operatorname{Map}((\mathcal{G}/U)_r,\Rat_p/\Int_p))=0.
$$
\end{proof}

\section{\texorpdfstring{$L$}{L}-functions of perfect complexes of adic sheaves}

We will briefly recall some  notation from \cite{Witte:NoncommutativeLFunctions} (see also \cite{Witte:PhD} and \cite{Witte:MCVarFF}). Let $\FF$ be the finite field with $q=p^v$ elements and fix an algebraic closure $\algc{\FF}$ of $\FF$. For any scheme $X$ in in the category $\cat{Sch}^\sep_{\FF}$ of separated schemes of finite type over $\FF$ and any adic ring $\Lambda$
we introduced a Waldhausen category
$\cat{PDG}^{\cont}(X,\Lambda)$ of perfect complexes of adic
sheaves on $X$ \cite[Def.~4.3]{Witte:NoncommutativeLFunctions}. The objects of this category are certain inverse systems over the index set $\openideals_{\Lambda}$ such that for $I\in\openideals_{\Lambda}$ the $I$-th level is a perfect complex of \'etale sheaves of $\Lambda/I$-modules.

Let us write $\KTh_0(X,\Lambda)$ for the zeroth Waldhausen $\KTh$-group of $\cat{PDG}^{\cont}(X,\Lambda)$, i.\,e.\ the abelian group generated by quasi-isomorphism classes of objects in the category $\cat{PDG}^{\cont}(X,\Lambda)$ modulo the relations
$$
[\cmplx{\sheaf{G}}][\cmplx{\sheaf{F}}]^{-1}[\cmplx{\sheaf{H}}]^{-1}
$$
for any sequence
$$
0\mto\cmplx{\sheaf{F}}\mto\cmplx{\sheaf{G}}\mto\cmplx{\sheaf{H}}\mto 0
$$
in $\cat{PDG}^{\cont}(X,\Lambda)$ which is exact (in each level $I\in\openideals_{\Lambda}$).

For any morphism $f\colon X\mto Y$ in $\cat{Sch}^\sep_{\FF}$ we have Waldhausen exact functors
\begin{align*}
f^*\colon \cat{PDG}^{\cont}(Y,\Lambda)&\mto \cat{PDG}^{\cont}(X,\Lambda),\\
\RDer f_!\colon\cat{PDG}^{\cont}(X,\Lambda)&\mto \cat{PDG}^{\cont}(Y,\Lambda)
\end{align*}
that correspond to the usual inverse image and the direct image with proper support. As Waldhausen exact functors they induce homomorphisms
$$
f^*\colon \KTh_0(Y,\Lambda)\mto\KTh_0(X,\Lambda),\qquad \RDer f_!\colon \KTh_0(X,\Lambda)\mto\KTh_0(Y,\Lambda).
$$
If $\Lambda'$ is a second adic ring we let $\Lambda^\op$-$\cat{SP}(\Lambda')$ denote the Waldhausen category of complexes of $\Lambda'$-$\Lambda$-bimodules which are strictly perfect as complexes of $\Lambda'$-modules. For each such complex $\cmplx{P}$ we have a change-of-ring homomorphism
$$
\ringtransf_{\cmplx{P}}\colon \KTh_0(X,\Lambda)\mto\KTh_0(X,\Lambda').
$$
The compositions of these homomorphisms behave as expected. In particular, $\ringtransf_{\cmplx{P}}$ commutes with $f^*$ and $\RDer f_!$, for $f\colon X\mto Y$, $g\colon Y\mto Z$ we have $(g\comp f)^*=f^*\comp g^*$, $\RDer (f\comp g)_!=\RDer f_!\comp \RDer g_!$, and $\RDer f'_!g'^*=g^*\RDer f_!$ if
$$
\xymatrix{W\ar[r]^{f'}\ar[d]^{g'}&X\ar[d]^{g}\\
          Y\ar[r]^{f}            &Z}
$$
is a cartesian square \cite[\S~4]{Witte:NoncommutativeLFunctions}.

For any $A\in\KTh_0(X,\Lambda)$ we define the $L$-function $L(A,T)\in\KTh_1(\Lambda[[T]])$ as follows. First, assume that $X=\Spec \FF'$ for a finite field extension $\FF'/\FF$ of degree $d$, that $\Lambda$ is finite and that $A$ is the class of a locally constant flat \'etale sheaf of $\Lambda$-modules on $X$. This sheaf corresponds to a finitely generated, projective $\Lambda$-module $P$ with a continuous action of the absolute Galois group of $\FF'$, which is topologically generated by the geometric Frobenius automorphism $\Frob_{\FF'}$ of $\FF'$. We then let $L(A,T)$ be the inverse of the class of the automorphism $\id-\Frob_{\FF'}T^d$ of $\Lambda[[T]]\tensor_{\Lambda}P$ in $\KTh_1(\Lambda[[T]])$. If $A$ is the class of any perfect complex of \'etale sheaves of $\Lambda$-modules on $\Spec \FF'$, we replace it by a quasi-isomorphic, strictly perfect complex $\cmplx{\sheaf{P}}$ and define $L(A,T)$ as the alternating product
$$
L(A,T)=\prod_k L(\sheaf{P}^k,T)^{(-1)^k}.
$$
This then extends to a group homomorphism $\KTh_0(\Spec \FF',\Lambda)\mto\KTh_1(\Lambda[[T]])$. If $\Lambda$ is an arbitrary adic ring and $A\in\KTh_0(\Spec \FF',\Lambda)$, then $L(A,T)$ is given by the system $(L(\ringtransf_{A/I}(A),T))_{I\in\openideals_{\Lambda}}$ in
$$
\KTh_1(\Lambda[[T]])=\varprojlim_{I\in\openideals_{\Lambda}}\KTh_1(\Lambda/I[[T]]).
$$
Finally, if $X$ is any separated scheme over $\FF$, we let $X^0$ denote the set of closed points of $X$ and
$$
x\colon \Spec k(x)\mto X
$$
the closed immersion corresponding to any $x\in X^0$. We set
$$
L(A,T)=\prod_{x\in X^0}L(x^*(A),T)\in\KTh_1(\Lambda[[T]]).
$$
The product converges in the topology of $\KTh_1(\Lambda[[T]])$ induced by the adic topology of $\Lambda[[T]]$ because $T\in\Jac(\Lambda[[T]])$ and for any $d$ there are only finitely many closed points of degree $d$ in $X$. We have thus constructed a group homomorphism
$$
L\colon \KTh_0(X,\Lambda)\mto \KTh_1(\Lambda[[T]]),\qquad A\mapsto L(A,T),
$$
for any $X$ in $\cat{Sch}^\sep_\FF$ and any adic ring $\Lambda$.

This construction agrees with \cite[Def.~6.4]{Witte:NoncommutativeLFunctions}. If $\Lambda$ is commutative, such that $\KTh_1(\Lambda[[T]])\isomorph \Lambda[[T]]^\times$ via the determinant map, it is also seen to agree with the classical definition used in \cite{EmertonKisin}. Moreover, we note that for any pair of adic rings $\Lambda$ and $\Lambda'$ and any complex $\cmplx{P}$ of bimodules in $\Lambda^\op$-$\cat{SP}(\Lambda')$, we have
$$
L(\ringtransf_{\cmplx{P}}(A),T)=\ringtransf_{\cmplx{P[[T]]}}(L(A,T))
$$
for $A\in\KTh_1(X,\Lambda)$.

\begin{rem}\label{rem:dependency on base field}
Note that $L(A,T)$ depends on the base field $\FF$. If $\FF'\subset\FF$ is a subfield, then the $L$-function of $A$ with respect to $\FF'$ is $L(A,T^{[\FF:\FF']})$ \cite[Rem.~6.5]{Witte:NoncommutativeLFunctions}.
\end{rem}

\section{The construction of unit \texorpdfstring{$L$}{L}-functions}

In this section, we prove the following theorem.

\begin{thm}\label{thm:construction of Q}
Let $s\colon X\mto \Spec \FF$ be a separated $\FF$-scheme of finite type. There exists a unique way to associate to each adic $\Int_p$-algebra $\Lambda$ a homomorphism
$$
Q\colon\KTh_0(X,\Lambda)\mto\widehat{\KTh}_1(\talg{\Lambda}),\qquad A\mapsto Q(A,T),
$$
such that for $A\in\KTh_0(X,\Lambda)$
\begin{enumerate}
\item the image of $Q(A,T)$ in $\KTh_1(\Lambda[[T]])$ is the ratio $L(A,T)/L(\RDer s_!A,T)$,
\item if $\Lambda'$ is a second adic ring and $\cmplx{P}$ is in ${\Lambda'}^\op$-$\cat{SP}(\Lambda)$, then
$$
\ringtransf_{\cmplx{\talg{P}}}(Q(A,T))=Q(\ringtransf_{\cmplx{P}}(A),T).
$$
\end{enumerate}
\end{thm}
\begin{proof}
If we restrict to the class of adic rings $\Lambda$ which are full matrix algebras over commutative adic $\Int_p$-algebras, then the natural homomorphism $\widehat{\KTh}_1(\talg{\Lambda})\mto\KTh_1(\Lambda[[T]])$ is injective and the existence of $Q\colon\KTh_0(X,\Lambda)\mto\widehat{\KTh}_1(\talg{\Lambda})$ follows from \cite[Cor.~1.8]{EmertonKisin} and Morita invariance.  In fact, we even know that the image of $Q$ lies in the subgroup $\KTh_1(\talg{\Lambda},\Jac(\Lambda)T\talg{\Lambda})$. (The result of Emerton and Kisin is stated only for $\FF=\FF_p$, but the result for general $\FF$ follows from Remark~\ref{rem:dependency on base field} and the simple observation that $\lambda(T)\in\Lambda[[T]]$ is in $\talg{\Lambda}$ if and only if for some $n>0$ $\lambda(T^n)\in\talg{\Lambda}$.) We note further that it suffices to prove the assertion of the theorem for the class of finite $\Int_p$-algebras. Then general case then follows by taking projective limits.

We proceed by induction on $\dim X$. So, we assume that the theorem has already been proved for all schemes in $\cat{Sch}^\sep_{\FF}$ of dimension less than $\dim X$. Note that any open subscheme $j\colon U\mto X$ with closed complement $i\colon Z\mto X$ induces a decomposition
$$
\KTh_0(U,\Lambda)\times\KTh_0(Z,\Lambda)\xrightarrow{\isomorph}\KTh_0(X,\Lambda), \qquad (A,B)\mapsto (\RDer j_!A)(\RDer i_!B)
$$
and if the theorem is true for $U$ and $Z$, then it is also true for $X$. In particular, we may reduce to the case that $X$ is an integral scheme.

If $\Lambda$ is finite then each object $\cmplx{\sheaf{F}}$ in $\cat{PDG}^{\cont}(X,\Lambda)$ is quasi-isomorphic to a strictly perfect complex $\cmplx{\sheaf{P}}$ of \'etale sheaves of $\Lambda$-modules on $X$. This means, $\sheaf{P}^n$ is flat and constructible and for $|n|$ sufficiently large, $\sheaf{P}^n=0$. In particular, we may choose $j\colon U\mto X$ open and dense such that $j^*\sheaf{P}^n$ is locally constant for each $n$. We may then find a finite connected Galois covering $g\colon V\mto U$ with Galois group $G$ and a complex $\cmplx{P}$ in $\Int_p[G]^\op$-$\cat{SP}(\Lambda)$ such that
$$
j^*\cmplx{\sheaf{P}}\isomorph\ringtransf_{\cmplx{P}}(g_!g^*\Int_p)
$$
\cite[Lemma~4.12]{Witte:NoncommutativeLFunctions}. Here, we consider $g_!g^*\Int_p$ as an object in $\cat{PDG}^{\cont}(U,\Int_p[G])$. Since the function field of $X$ is of characteristic $p$, the cohomological $p$-dimension of its absolute Galois group $1$. So, we may apply Prop.~\ref{prop:vanishing of SK} and choose $G$ (after possibly shrinking $U$) large enough such that
$$
\ringtransf_{\cmplx{\talg{P}}}\colon\KTh_0(\talg{\Int_p[G]})\mto\KTh_0(\talg{\Lambda})
$$
factors through $\Det(\talg{\Int_p[G]}^\times)$.

Let $s\colon U\mto \Spec \FF$ be the structure map. We will now show that
$$
\alpha=L(g_!g^*\Int_p,T)/L(\RDer s_!(g_!g^*\Int_p),T)\in\KTh_1(\Int_p[G][[T]])
$$
has a preimage in $\KTh_1(\talg{\Int_p[G]})$. First, we note that by construction, $\alpha$ is in the subgroup $\KTh_1(\Int_p[G][[T]],T\Int_p[G][[T]])$. Further, by \cite[Fonction $L$ mod $\ell^n$, Theorem~2.2.(b)]{SGA4h}, we know that the image of $\alpha$ in $\KTh_1(\Int_p[G]/\Jac(\Int_p[G])[[T]])$ vanishes. Hence, using Lemma~\ref{lem:SK1 for power series rings}, we may view $\alpha$ as an element of the group $\Det(1+\Jac(\Int_p[G])T\Int_p[G][[T]])$. From \cite[Cor.~1.8]{EmertonKisin} and Lemma~\ref{lem:p torsion criterion} we conclude that $\alpha^{p^n}$ is in $\Det(\talg{\Int_p[G]}^\times)$ for some $n\geq 0$. By Lemma~\ref{lem:induction} it thence suffices to show that for every $\Rat_p$-$p$-elementary subgroup $H\subset G$, the element
$$
\Res_G^H \alpha=L(\Res_G^H g_!g^*\Int_p,T)/L(\RDer s_!\Res_G^H g_!g^*\Int_p,T)\in\KTh_1(\Int_p[H][[T]])
$$
has a preimage in $\KTh_1(\talg{\Int_p[H]})$. Choosing the ideal $A\subset \Int_p[H]$ as in Lemma~\ref{lem:p elementary case} and noting that by part $(1)$ of this lemma, $\Int_p[H]/A$ is a finite product of full matrix rings over commutative adic rings, we can again apply \cite[Cor.~1.8]{EmertonKisin} to see that the image of $\Res_G^H \alpha$ in
$$
\KTh_1(\Int_p[H]/A[[T]])/\KTh_1(\talg{\Int_p[H]/A})=\Det(\Int_p[H]/A[[T]]^\times)/\Det(\talg{\Int_p[H]/A}^\times)
$$
vanishes. From the exact sequence
\begin{multline*}
\Det(1+A\Int_p[H][[T]])/\Det(1+A\talg{\Int_p[H]})\mto\Det(\Int_p[H][[T]]^\times)/\Det(\talg{\Int_p[H]}^\times)\\
\mto\Det(\Int_p[H]/A[[T]]^\times)/\Det(\talg{\Int_p[H]/A}^\times)\mto 0
\end{multline*}
and part $(3)$ of Lemma~\ref{lem:p elementary case} we conclude that $\Res_G^H \alpha$ does indeed have a preimage in $\KTh_1(\talg{\Int_p[H]})$. Thus, the element $\alpha$ has a preimage $\widetilde{\alpha}$ in $\KTh_1(\talg{\Int_p[G]})$.

We return to the complex $\cmplx{\sheaf{F}}$ and define
$$
Q(\cmplx{\sheaf{F}},T)=Q(i^*\cmplx{\sheaf{F}},T)\cdot\ringtransf_{\cmplx{\talg{P}}}(\widetilde{\alpha}).
$$
We need to check that this definition does not depend on the choices of $U$, $V$, $\cmplx{P}$, and $\widetilde{\alpha}$. For this, let $j'\colon U'\mto U$ be an open dense subscheme with closed complement $i'\colon Z'\mto U$. Then $g\colon U'\times_U V\mto U'$ is still a finite connected Galois covering with Galois group $G$. Let $h\colon V'\mto U'\times_U V$ be a finite connected Galois covering with Galois group $H$ and assume that $g'=g\comp h$ is Galois with Galois group $G'$ such that $G'/H=G$. Assume further that $\cmplx{P'}$ is a complex in $\Int_p[G']^\op$-$\cat{SP}(\Lambda)$ such that $
\ringtransf_{\cmplx{P'}}(g'_!g'^*\Int_p)$ is quasi-isomorphic to $(j'\comp j)^*\cmplx{\sheaf{F}}$. By taking the stalks of $\ringtransf_{\cmplx{P'}}(g'_!g'^*\Int_p)$ and $\ringtransf_{\cmplx{P}}(g_!g^*\Int_p)$ at any geometric point of $U'$, we see that $\cmplx{P'}\tensor_{\Int_p[G']}\Int_p[G]$ is quasi-isomorphic to $\cmplx{P}$ in $\Int_p[G]^\op$-$\cat{SP}(\Lambda)$. In particular,
$$
\ringtransf_{\cmplx{\talg{P'}}}=\ringtransf_{\cmplx{\talg{P}}}\comp\ringtransf_{\talg{\Int_p[G]}}\colon \KTh_1(\talg{\Int_p[G']})\mto\KTh_1(\talg{\Lambda}).
$$
As above, we construct a preimage $\widetilde{\alpha}'$ of $L(g'_!g'^*\Int_p,T)/L(\RDer(s\comp j')_!(g'_!g'^*\Int_p),T)$ in $\KTh_1(\talg{\Int_p[G']})$. Since
$$
\Det(\talg{\Int_p[G]}^\times)\mto\Det(\Int_p[G][[T]]^\times)
$$
is injective, we see that the image of $\widetilde{\alpha}/\ringtransf_{\talg{\Int_p[G]}}(\widetilde{\alpha}')$ in $\Det(\talg{\Int_p[G]}^\times)$ must agree with the unique preimage of $L(i'^*g_!g^*\Int_p,T)/L(\RDer (s\comp i')_! i'^*g_!g^*\Int_p,T)$, which in turn agrees with the image of $Q(i'^*g_!g^*\Int_p,T)$ by the induction hypothesis. Hence,
$$
\ringtransf_{\cmplx{\talg{P}}}(\widetilde{\alpha})=\ringtransf_{\cmplx{\talg{P}}}(Q(i'^*g_!g^*\Int_p,T)\ringtransf_{\talg{\Int_p[G]}}(\widetilde{\alpha}'))
=Q(i'^*\cmplx{\sheaf{F}},T)\ringtransf_{\cmplx{P'}}(\widetilde{\alpha}').
$$
Let $i''\colon Z''\mto X$ be the closed complement of $U'$ in $X$. By the uniqueness of $Q$ for $Z''$ we conclude
$$
Q(i''^*\cmplx{\sheaf{F}},T)=Q(i'^*\cmplx{\sheaf{F}},T)Q(i^*\cmplx{\sheaf{F}},T)
$$
and so, the above definition of $Q(\cmplx{\sheaf{F}},T)$ is independent of all choices.

Assume now that $\Lambda'$ is a second finite $\Int_p$-algebra and that $\cmplx{W}$ is in $\Lambda^\op$-$\cat{SP}(\Lambda')$. Then
\begin{align*}
\ringtransf_{\cmplx{\talg{W}}}(Q(\cmplx{\sheaf{F}},T))&=\ringtransf_{\cmplx{\talg{W}}}(Q(i^*\cmplx{\sheaf{F}},T)\ringtransf_{\cmplx{\talg{P}}}(\widetilde{\alpha}))\\
&=Q(i^*\ringtransf_{\cmplx{W}}(\cmplx{\sheaf{F}}),T)\ringtransf_{W\tensor_\Lambda \cmplx{\talg{P}}}(\widetilde{\alpha})\\
&=Q(\ringtransf_{\cmplx{W}}(\cmplx{\sheaf{F}}),T).
\end{align*}

It is immediately clear from the definition that $Q(\cmplx{\sheaf{F}},T)$ does only depend on the quasi-isomorphism class of $\cmplx{\sheaf{F}}$. Moreover, if $\Lambda$ is finite, any exact sequence in $\cat{PDG}^{\cont}(X,\Lambda)$ can be completed to a diagram
$$
\xymatrix{
0\ar[r]      &\cmplx{\sheaf{P}}_1\ar[r]\ar[d]&\cmplx{\sheaf{P}}_2\ar[r]\ar[d]&\cmplx{\sheaf{P}}_3\ar[r]\ar[d]& 0\\
0\ar[r]      &\cmplx{\sheaf{F}}_1\ar[r]      &\cmplx{\sheaf{F}}_2\ar[r]      &\cmplx{\sheaf{F}}_3\ar[r]      & 0\\
}
$$
where the top row is an exact sequence of strictly perfect complexes, the bottom row is the original exact sequence, the downward pointing arrows are quasi-isomorphisms, and the squares commute up to homotopy. 
Then one finds $j\colon U\mto X$, $i\colon Z\mto X$, $g\colon V\mto U$, $G$ as above and an exact sequence
$$
0\mto \cmplx{P}_1\mto\cmplx{P}_2\mto\cmplx{P}_3\mto 0
$$
of complexes in $\Int_p[G]^\op$-$\cat{SP}(\Lambda)$ such that for $k=1,2,3$,
$$
\ringtransf_{\cmplx{P}_k}(g_!g^*\Int_p)\isomorph j^*\cmplx{\sheaf{P}}_k.
$$
Choosing $\widetilde{\alpha}$ as above, we conclude
\begin{align*}
Q(\cmplx{\sheaf{F}}_2,T)&=Q(i^*\cmplx{\sheaf{F}}_1,T)Q(i^*\cmplx{\sheaf{F}}_3,T)\ringtransf_{\cmplx{\talg{P}_1}}(\widetilde{\alpha})\ringtransf_{\cmplx{\talg{P}_3}}(\widetilde{\alpha})\\
                        &=Q(\cmplx{\sheaf{F}}_1,T)Q(\cmplx{\sheaf{F}}_3,T).
\end{align*}
Thus, $Q$ is well-defined as homomorphism $\KTh_0(X,\Lambda)\mto\KTh_1(\talg{\Lambda})$ and satisfies properties $(1)$ and $(2)$ of the theorem.

It remains to prove uniqueness. So let $Q'$ be a second family of homomorphisms $\KTh_0(X,\Lambda)\mto\widehat{\KTh}_1(\talg{\Lambda})$ ranging over all adic rings $\Lambda$ such that $(1)$ and $(2)$ are satisfied. If $i\colon Z\mto X$ is any closed subscheme, then
$$
\KTh_0(Z,\Lambda)\mto\widehat{\KTh}_1(\talg{\Lambda}),\qquad A\mapsto Q'(i_*A,T),
$$
still satisfies $(1)$ and $(2)$. Thus, by the induction hypothesis, we must have
$$
Q'(i_*A,T)=Q(A,T).
$$
If $\Lambda$ is finite and $G$ is a finite group and $\cmplx{P}$ a complex in $\Int_p[G]^\op$-$\cat{SP}(\Lambda)$ such that $\ringtransf_{\cmplx{\talg{P}}}$ factors through $\Det(\talg{\Int_p[G]})$, then by the injectivity of
$$
\Det(\talg{\Int_p[G]})\mto\Det(\Int_p[G][[T]])
$$
and properties $(1)$ and $(2)$ we must have
$$
Q'(\ringtransf_{\cmplx{\talg{P}}}(B),T)=Q(\ringtransf_{\cmplx{\talg{P}}}(B),T)
$$
for all $B\in\KTh_0(X,\Int_p[G])$. By the construction of $Q$ we thus see that $Q=Q'$. This completes the proof of the theorem.
\end{proof}

We conclude with some ancillary observations. First, we note that $Q$ depends in the same way as the $L$-function on the choosen base field $\FF$:

\begin{prop}\label{prop:dependence on the base field}
Let $X$ be a separated scheme of finite type over $\FF$ and $\FF'\subset \FF$ be a subfield. Write $Q'\colon\KTh_0(X,\Lambda)\mto\widehat{\KTh}_1(\talg{\Lambda})$ for the homomorphism resulting from applying Thm~\ref{thm:construction of Q} to $X$ considered as $\FF'$-scheme. Then
$$
Q'(A,T)=Q(A,T^{[\FF':\FF]})
$$
for every $A\in \KTh_0(X,\Lambda)$.
\end{prop}
\begin{proof}
By Remark~\ref{rem:dependency on base field},  the given equality holds for the images in $\KTh_1(\Lambda[[T]])$. Now one proceeds as in the proof of the uniqueness part of Thm~\ref{thm:construction of Q} to show that it also holds in $\widehat{\KTh}_1(\talg{\Lambda})$.
\end{proof}

As in \cite{EmertonKisin}, we can also define a version of $Q$ relative to a morphism $f\colon X\mto Y$ in $\cat{Sch}^\sep_{\FF}$ by setting
$$
Q(f)\colon\KTh_0(X,\Lambda)\mto\widehat{\KTh}_1(\talg{\Lambda}),\qquad A\mapsto Q(f,A,T)=Q(A,T)/Q(\RDer f_!A,T),
$$
and extend \cite[Lemma~2.4, 2.5]{EmertonKisin} as follows.

\begin{prop}
Let $A$ be an object in $\KTh_0(X,\Lambda)$.
\begin{enumerate}
 \item If $f\colon X\mto Y$, $g\colon Y\mto Z$ are two morphisms in $\cat{Sch}^\sep_{\FF}$, then
       $$
       Q(g\comp f,A,T)=Q(g,\RDer f_!A,T)Q(f,A,T).
       $$
 \item If $f\colon X\mto Y$ is quasi-finite, then
       $$
       Q(f,A,T)=1.
       $$
\end{enumerate}
\end{prop}
\begin{proof}
Assertion $(1)$ follows immediately from the definition. For the second assertion, we note that for $f\colon X\mto Y$ quasi-finite,
\begin{align*}
L(\RDer f_!A,T)&=\prod_{y\in Y^0}L(y^*f_!A,T)\\
               &=\prod_{y\in Y^0}\prod_{x\in f^{-1}(y)^0}L(x^*A,T)\\
               &=\prod_{x\in X^0}L(x^*A,T)=L(A,T).
\end{align*}
In particular, the images of $Q(A,T)$ and $Q(\RDer f_!A,T)$ in $\KTh_1(\Lambda[[T]])$ agree. By the uniqueness of $Q$, we must have
$$
Q(A,T)=Q(\RDer f_!A,T)
$$
in $\widehat{\KTh}_1(\talg{\Lambda})$.
\end{proof}

Finally, setting
$$
\widehat{\KTh}_0(X,\Lambda)=\varprojlim_{I\in\openideals_{\Lambda}}\KTh_0(X,\Lambda/I),
$$
we observe that the constructions of the $L$-function and of $Q$ extends to
$$
L\colon\widehat{\KTh}_0(X,\Lambda)\mto\KTh_1(\Lambda[[T]]), \qquad Q\colon\widehat{\KTh}_0(X,\Lambda)\mto\widehat{\KTh}_1(\talg{\Lambda}).
$$
We cannot say much about the canonical homomorphism $\KTh_0(X,\Lambda)\mto\widehat{\KTh}_0(X,\Lambda)$, but we suspect that it might be neither injective nor surjective in general.

\section{A noncommutative main conjecture for separated schemes over \texorpdfstring{$\FF$}{F}}\label{sec:MC}

In this section, we will complete the $\ell=p$ case of the version of the noncommutative Iwasawa main conjecture for separated schemes $X$ over $\FF=\FF_q$ considered in \cite{Witte:MCVarFF}. We will briefly recall the terminology of the cited article.

Recall from \cite[Def.~2.1]{Witte:MCVarFF} that a principal covering $(f\colon Y\mto X,G)$ of $X$ with $G$ a profinite group is an inverse system $(f_U\colon Y_U\mto X)_{U\in\cat{NS}(G)}$ of finite principal $G/U$-coverings (not necessarily connected), where $U$ runs through the lattice $\cat{NS}(G)$ of open normal subgroups of $G$. As a particular case, for $k$ prime to $p$ and $\varGamma_{kp^\infty}=\Gal(\FF_q^{kp^{\infty}}/\FF_q)$, we have the cyclotomic $\varGamma_{kp^\infty}$-covering
$$
(X_{kp^\infty}=X\times_{\Spec \FF_q}\Spec \FF_{q^{kp^\infty}}\mto X,\varGamma_{kp^\infty})
$$
\cite[Def.~2.5]{Witte:MCVarFF}. We will only consider principal coverings $(f\colon Y\mto X,G)$  such that there exists a closed normal subgroub $H\subset G$ which is a topologically finitely generated virtual pro-$p$-group and such that the quotient covering $(f_H\colon Y_H\mto X,G/H)$ is the cyclotomic $\varGamma_{p^\infty}$-covering. These coverings will be called admissible coverings for short \cite[Def.~2.6]{Witte:MCVarFF}. For such a group $G=H\rtimes \varGamma_{p^\infty}$, if $\Lambda$ is any adic $\Int_p$-algebra, then the profinite group ring $\Lambda[[G]]$ is again an adic $\Int_p$-algebra \cite[Prop.~3.2]{Witte:MCVarFF}.

For any admissible covering $(f\colon Y\mto X,G)$  we constructed in \cite[Prop.~6.2]{Witte:MCVarFF} a Waldhausen exact functor
$$
f_!f^*\colon\cat{PDG}^\cont(X,\Lambda)\mto\cat{PDG}^\cont(X,\Lambda[[G]])
$$
by forming the inverse system over the intermediate finite coverings. As before, we will denote the induced homomorphism
$$
f_!f^*\colon\KTh_0(X,\Lambda)\mto \KTh_0(X,\Lambda[[G]])
$$
by the same symbol.

We also constructed a localisation sequence
$$
1\mto\KTh_1(\Lambda[[G]])\mto\KTh_1^{\loc}(H,\Lambda[[G]])\xrightarrow{d} \KTh_0^{\rel}(H,\Lambda[[G]])\mto 1
$$
with
\begin{align*}
\KTh_1^{\loc}(H,\Lambda[[G]])&=\KTh_1(w_H\cat{PDG}^{\cont}(\Lambda[[G]])),\\
 \KTh_0^{\rel}(H,\Lambda[[G]])&=\KTh_0(\cat{PDG}^{\cont,w_H}(\Lambda[[G]]))
\end{align*}
and certain Waldhausen categories $w_H\cat{PDG}^\cont(\Lambda[[G]])$ and $\cat{PDG}^{\cont,w_H}(\Lambda[[G]])$, respectively  \cite[\S~4]{Witte:MCVarFF}, \cite[Cor.~3.3]{Witte:Splitting}. Whenever $\Lambda[[H]]$ is noetherian, there exists a left denominator set $S$ such that
\begin{align*}
\KTh_1^{\loc}(H,\Lambda[[G]])&=\KTh_1(\Lambda[[G]]_S),\\
\KTh_0^{\rel}(H,\Lambda[[G]])&=\KTh_0(\Lambda[[G]],\Lambda[[G]]_S),
\end{align*}
where $\Lambda[[G]]_S$ is the localisation of $\Lambda[[G]]$ at $S$ and the $\KTh$-groups on the right-hand side are the usual groups appearing in the corresponding localisation sequence of higher $\KTh$-theory \cite[Thm.~2.18, Prop.~2.20, Rem~2.22]{Witte:Splitting}. In particular, if $\Lambda$ is commutative (hence, noetherian) and $G=\varGamma_{kp^\infty}$, then the Ore set $S$ is given by
$$
S=\set{\lambda\in\Lambda[[\varGamma_{kp^\infty}]]\colon\text{$\lambda$ is a nonzerodivisor in $\Lambda/\Jac(\Lambda)[[\varGamma_{kp^\infty}]]$}}
$$
and
$$
\KTh_1^{\loc}(H,\Lambda[[G]])=\KTh_1(\Lambda[[G]]_S)=\Lambda[[\varGamma_{kp^\infty}]]_S^\times.
$$

Let $\Lambda'$ be a second adic ring and $\cmplx{M}$ a complex in $\Lambda[[G]]^\op\mathrm{-}\cat{SP}(\Lambda')$. Then
$$
\cmplx{\widetilde{M}}=\ringtransf_{\cmplx{M}}(f_!f^*\Lambda)
$$
is a perfect complex of smooth $\Lambda'$-adic sheaves on $X$ with an additional right $\Lambda$-module structure. Using this complex we obtain a homomorphism
$$
\ringtransf_{\cmplx{\widetilde{M}}}\colon \KTh_0(X,\Lambda)\mto\KTh_0(X,\Lambda').
$$
Furthermore, we can form the complex
$$
\cmplx{M[[G]]^\delta}=\Lambda'[[G]]\tensor_{\Lambda}\cmplx{M}
$$
in $\Lambda[[G]]^\op$-$\cat{SP}(\Lambda[[G]])$ with the canonical left $G$-action and the diagonal right $G$-action.

If $G$ acts trivially on $\cmplx{M}$, then $\ringtransf_{\cmplx{\widetilde{M}}}$ is just the homomorphism
$$
\ringtransf_{\cmplx{M}}\colon \KTh_0(X,\Lambda)\mto\KTh_0(X,\Lambda')
$$
that we have already used above. Moreover, we have
$$
\ringtransf_{\cmplx{M[[G]]^\delta}}f_!f^*A=f_!f^*\ringtransf_{\cmplx{\widetilde{M}}}(A)
$$
\cite[Prop.~6.7]{Witte:MCVarFF}. Note that $\cmplx{M[[G]]^\delta}$ also induces compatible homomorphisms on $\KTh_1(\Lambda[[G]])$, $\KTh_1^\loc(H,\Lambda[[G]])$ and $\KTh_0^\rel(H,\Lambda[[G]])$ \cite[Prop.~4.6]{Witte:MCVarFF}.

As a special case we may take $\Lambda=\Lambda'=\Int_p$ and let $\rho\colon G\mto \operatorname{GL}_n(\Int_p)$  be a continuous left $G$-representation as in the introduction. We may then choose $M$  to be the $\Int_p$-$\Int_p[[G]]$-module obtained from $\rho$ by letting $G$ act contragrediently on $\Int_p^n$ from the right. The sheaf $\widetilde{M}$ is then just the smooth $\Lambda$-adic sheaf $\sheaf{M}(\rho)$ associated to $\rho$ and $\ringtransf_{\widetilde{M}}$ corresponds to taking the (completed) tensor product with this sheaf over $\Int_p$.

In \cite[Thm.~8.1]{Witte:MCVarFF} we have already shown that for each complex $\cmplx{\sheaf{F}}$ in the Waldhausen category $\cat{PDG}^{\cont}(X,\Lambda)$ and each admissible covering $(f\colon Y\mto X,G)$ the complex of $\Lambda$-adic cohomology with proper support
$$
\RDer \Sectc(X,f_!f^*\cmplx{\sheaf{F}})
$$
is an object of $\cat{PDG}^{\cont,w_H}(\Lambda[[G]])$. Hence, we obtain a homomorphism
$$
\KTh_0(X,\Lambda)\mto\KTh_0^{\rel}(H,\Lambda[[G]]),\qquad A\mapsto \RDer\Sectc(X,f_!f^*A).
$$
We have also contructed an explicit homomorphism
$$
\KTh_0(X,\Lambda)\mto\KTh_1^{\loc}(H,\Lambda[[G]]),\qquad A\mapsto \ncL_G(X/\FF,A),
$$
such that
$$
d \ncL_G(X/\FF,A)=\RDer\Sect_c(X,f_!f^*A)^{-1}
$$
\cite[Def.~8.3]{Witte:MCVarFF}.

We let $\gamma$ denote the image of the geometric Frobenius automorphism $\Frob_\FF$ in $\varGamma_{kp^\infty}$. If $\Omega$ is a commutative adic $\Int_p$-algebra, we write $\widetilde{S}\subset\talg{\Omega}$ for the denominator set consisting of those elements which become a unit in $\Omega[[T]]$. The proof of \cite[Lemma~8.5]{Witte:MCVarFF} shows that the evaluation $T\mapsto\gamma^{-1}$ extends to a ring homomorphism
$$
\talg{\Omega}_{\widetilde{S}}\mto \Omega[[\varGamma_{kp^\infty}]]_S,\qquad \omega(T)\mapsto \omega(\gamma^{-1}).
$$
Note that $\talg{\Omega}_{\widetilde{S}}$ is a semilocal ring, hence
$$
\talg{\Omega}_{\widetilde{S}}^\times=\KTh_1(\talg{\Omega}_{\widetilde{S}}).
$$

Let $s\colon X\mto \Spec \FF$ be the structure map. Then the proof of \cite[Thm.~8.6]{Witte:MCVarFF} shows that
for every $\cmplx{M}$ in $\Lambda[[G]]^\op$-$\cat{SP}(\Omega)$ and every $A\in\KTh_0(X,\Lambda)$, we have
$$
L\left(\RDer s_!\ringtransf_{\cmplx{\widetilde{M}}}(f_!f^*A),T\right)\in\KTh_1(\talg{\Omega}_{\widetilde{S}})
$$
and that
$$
\ringtransf_{\Omega[[\varGamma_{kp^\infty}]]}\ringtransf_{\cmplx{M[[G]]^\delta}}(\ncL_G(X/\FF,A))=L\left(\RDer s_!\ringtransf_{\cmplx{\widetilde{M}}}(f_!f^*A),\gamma^{-1}\right)
$$
in
$$
\KTh_1^{\loc}(H,\Omega[[\varGamma_{kp^\infty}]])=\Omega[[\varGamma_{kp^\infty}]]^\times_S.
$$

So, $\ncL_G(X/\FF,A)$ satisfies the desired interpolation property with respect to the functions $L\left(\RDer s_!\ringtransf_{\cmplx{\widetilde{M}}}(f_!f^*A),\gamma^{-1}\right)$, but not with respect to $L(\ringtransf_{\cmplx{\widetilde{M}}}(f_!f^*A),\gamma^{-1})$. We will construct a modification of $\ncL_G(X/\FF,A)$ below.

For any adic ring $\Lambda$, the evaluation map $T\mapsto 1$ induces a homomorphism
$$
\widehat{\KTh}_1(\talg{\Lambda})\mto \KTh_1(\Lambda),\qquad\lambda(T)\mapsto \lambda(1).
$$
We also obtain an evaluation map
$$
\widehat{\KTh}_1(\talg{\Lambda})\mto \KTh_1(\Lambda[[\varGamma_{kp^\infty}]]),\qquad \lambda(T)\mapsto \lambda(\gamma^{-1}),
$$
as composition of $T\mapsto 1$ with the automorphism of $\widehat{\KTh}_1(\talg{\Lambda[[\varGamma_{kp^\infty}]]})$ induced by $T\mapsto \gamma^{-1}T$ and the injection
$$
\ringtransf_{\talg{\Lambda[[\varGamma_{kp^\infty}]]}}\colon \widehat{\KTh}_1(\talg{\Lambda})\mto\widehat{\KTh}_1(\talg{\Lambda[[\varGamma_{kp^\infty}]]}).
$$

\begin{defn}
For any admissible covering $(f\colon Y\mto X,G)$ and any $A\in\KTh_0(X,\Lambda)$ we set
$$
\widetilde{\ncL}_G(X/\FF,A)=\ncL_G(X/\FF,A)Q(f_!f^*A,1).
$$
\end{defn}

Since $Q(f_!f^*A,1)\in\KTh_1(\Lambda[[G]])$, we still have

\begin{thm}
$$
d \widetilde{\ncL}_G(X/\FF,A)=\RDer\Sect_c(X,f_!f^*A)^{-1}
$$
in $\KTh_0^{\loc}(H,\Lambda[[G]])$.
\end{thm}

We will now investigate the transformation properties of $\widetilde{\ncL}_G(X/\FF,A)$.

\begin{thm}\label{thm:transformation properties}
Consider a separated scheme $X$ of finite type over a finite field $\FF$.
Let $\Lambda$ be any adic $\Int_p$ algebra and let $A$ be in $\KTh_0(X,\Lambda)$.
\begin{enumerate}
\item Let $\Lambda'$ be another adic $\Int_{p}$-algebra.
For any complex $\cmplx{M}$ in $\Lambda[[G]]^{\op}$-$\cat{SP}(\Lambda')$, we have
$$
\ringtransf_{\cmplx{M[[G]]^{\delta}}}(\widetilde{\ncL}_G(X/\FF,A))=\widetilde{\ncL}_G(X/\FF,\ringtransf_{\cmplx{\widetilde{M}}}(A))
$$
in $\KTh_1^\loc(H,\Lambda'[[G]])$.

\item Let $H'$ be a closed virtual pro-$p$-subgroup of $H$ which is normal in $G$. Then
$$
\ringtransf_{\Lambda[[G/H']]}(\widetilde{\ncL}_G(X/\FF,A))=\widetilde{\ncL}_{G/H'}(X/\FF,A)
$$
in $\KTh_1^\loc(H',\Lambda[[G/H']])$.

\item Let $U$ be an open subgroup of $G$ and let $\FF'$ be the finite extension corresponding to the image of $U$ in $\varGamma_{p^{\infty}}$. Then
$$
\ringtransf_{\Lambda[[G]]}\big(\widetilde{\ncL}_{G}(X/\FF,A)\big)=\widetilde{\ncL}_{U}(Y_U/\FF',f_U^*A)
$$
in $\KTh_1(H\cap U,(\Lambda[[U]]))$.
\end{enumerate}
\end{thm}
\begin{proof}
In \cite[Thm.~8.4]{Witte:MCVarFF} we have already proved that $\ncL_G(X/\FF,A)$ satisfies the given transformation properties. To prove the same properties for $Q(f_*f^*A,1)$, one uses the general transformation rule for $Q(A,T)$ and $\ringtransf$ and the equalities
\begin{align*}
\ringtransf_{\cmplx{M[[G]]^{\delta}}}f_!f^*A&=f_!f^*\ringtransf_{\cmplx{\widetilde{M}}}(A)\\
\ringtransf_{\Lambda[[G/H']]}f_!f^*A&={f_{H'}}_!f_{H'}^*A\\
\ringtransf_{\Lambda[[G]]}f_!f^*A&=\RDer {f_U}_!f^U_!{f^U}^* f_U^*A
\end{align*}
with $(f^U\colon Y\mto Y_U, U)$ the restriction of the covering to $U$
\cite[Prop.~6.5, 6.7]{Witte:MCVarFF}.  For $(3)$ it remains to notice that the evaluation $Q(A,1)$ does not depend on the base field $\FF$ by Prop.~\ref{prop:dependence on the base field}.
\end{proof}

\begin{prop}\label{prop:cyclotomic case}
Consider the admissible covering $(f\colon X_{kp^\infty}\mto X, \varGamma_{kp^\infty})$. For any adic ring $\Lambda$ and any $A\in\KTh_1(X,\Lambda)$,
$$
Q(f_!f^*A,T)=Q(\ringtransf_{\Lambda[[\varGamma_{kp^\infty}]]}(A),\gamma^{-1}T)
$$
in $\widehat{\KTh}_1(\talg{\Lambda[[\varGamma_{kp^\infty}]]})$.
\end{prop}
\begin{proof}
We may assume that $\Lambda$ is finite. Let $s\colon X\mto \FF$ denote the structure map. From \cite[Prop.~7.2]{Witte:MCVarFF} it follows that
$$
L(\RDer s_!f_!f^*A,T)=L(\RDer s_!\ringtransf_{\Lambda[[\varGamma_{kp^\infty}]]}(A),\gamma^{-1}T).
$$
By applying this to $x^*A$ for each closed point $x\colon \Spec k(x)\mto X$ of $X$ and using
$$
f_!f^*x^*A=x^*f_!f^*A
$$
\cite[Prop.~6.4.(1)]{Witte:MCVarFF}
we see that
$$
L(f_!f^*A,T)=L(\ringtransf_{\Lambda[[\varGamma_{kp^\infty}]]}(A),\gamma^{-1}T).
$$
Hence, the images of $Q(f_!f^*A,T)$ and $Q(\ringtransf_{\Lambda[[\varGamma_{kp^\infty}]]}(A),\gamma^{-1}T)$ agree in the group $\KTh_1(\Lambda[[\varGamma_{kp^\infty}]][[T]])$.

If $a\colon X'\mto X$ is a morphism in $\cat{Sch}^\sep_{\FF}$ and $(f\colon X'\mto X,\varGamma_{kp^\infty})$ is the cyclotomic $\varGamma_{kp^\infty}$-covering of $X'$, then
$$
\RDer a_!f_!f^*A=f_!f^*\RDer a_!A
$$
by \cite[Prop.~6.4.(2)]{Witte:MCVarFF}. In particular, this applies to open and closed immersions. By induction on the dimension of $X$ we can thus reduce to the case $X$ integral and $A=\ringtransf_{\cmplx{P}}(g_!g^*\Int_p)$ for some finite connected Galois covering $(g\colon Y\mto X, G)$ and some complex $\cmplx{P}$ in $\Int_p[G]^\op$-$\cat{SP}(\Lambda)$.
With the complex
$$
\cmplx{P[[\varGamma_{kp^\infty}]]}=\Lambda[[\varGamma_{kp^\infty}]]\tensor_{\Lambda}\cmplx{P}
$$
in $\Int_p[[G\times\varGamma_{kp^\infty}]]\text{-}\cat{SP}(\Lambda[[\varGamma_{kp^\infty}]])$ we have
\begin{align*}
\ringtransf_{\cmplx{P[[\varGamma_{kp^\infty}]]}}(f_!f^*g_!g^*\Int_p)&=f_!f^*A,\\ \ringtransf_{\cmplx{P[[\varGamma_{kp^\infty}]]}}(\ringtransf_{\Lambda[[\varGamma_{kp^\infty}]]}(g_!g^*\Int_p))&=\ringtransf_{\Lambda[[\varGamma_{kp^\infty}]]}(A).
\end{align*}
Applying Prop.~\ref{prop:vanishing of SK} to
$$
R=\talg{\Int_p[[\varGamma_{kp^\infty}]]}\isomorph\talg{\Int_p[\Int/k\Int][[T']]}
$$
we may assume that $\ringtransf_{\cmplx{\talg{P[[\varGamma_{kp^\infty}]]}}}$ factors through $\Det(\talg{\Int_p[[G\times\varGamma_{kp^\infty}]]})$. We conclude
\begin{align*}
Q(f_!f^*A,T)&=\ringtransf_{\cmplx{\talg{P[[\varGamma_{kp^\infty}]]}}}(Q(f_!f^*g_!g^*\Int_p,T))\\
            &=\ringtransf_{\cmplx{\talg{P[[\varGamma_{kp^\infty}]]}}}(Q(g_!g^*\Int_p,\gamma^{-1}T))\\
            &=Q(A,\gamma^{-1}T).
\end{align*}
\end{proof}

The following theorem shows that $\widetilde{\ncL}_{G}(X/\FF,A)$ satisfies the right interpolation property.

\begin{thm}\label{thm:link to classical L-function}
Let $X$ be a scheme in $\cat{Sch}^\sep_\FF$ and let $(f\colon Y\mto X,G)$ be an admissible principal covering containing the cyclotomic $\varGamma_{kp^{\infty}}$-covering.  Furthermore, let $\Lambda$ and $\Omega$ be adic $\Int_p$-algebras with $\Omega$ commutative. For every $A\in\KTh_0(X,\Lambda)$ and every $\cmplx{M}$ in $\Lambda[[G]]^\op$-$\cat{SP}(\Omega)$, we have
$$
L(\ringtransf_{\cmplx{\widetilde{M}}}(A),T)\in\KTh_1(\talg{\Omega}_{\widetilde{S}})
$$
and
$$
\ringtransf_{\Omega[[\varGamma_{kp^{\infty}}]]}\ringtransf_{\cmplx{M[[G]]^{\delta}}}\big(\widetilde{\ncL}_{G}(X/\FF,A)\big)= L(\ringtransf_{\cmplx{\widetilde{M}}}(A),\gamma^{-1})
$$
in $\KTh_1(\Omega[[\varGamma_{kp^{\infty}}]]_S)$.
\end{thm}
\begin{proof}
As remarked above, the corresponding statement for $L(\RDer s_!\ringtransf_{\cmplx{\widetilde{M}}}(A),T)$ and $\ncL_{G}(X/\FF,A)$ follows from \cite[Thm.~8.6]{Witte:MCVarFF}. Since
$$
L(\ringtransf_{\cmplx{\widetilde{M}}}(A),T)=Q(\ringtransf_{\cmplx{\widetilde{M}}}(A),T)L(\RDer s_!\ringtransf_{\cmplx{\widetilde{M}}}(A),T),
$$
an application of Prop.~\ref{prop:cyclotomic case} concludes the proof of the theorem.
\end{proof}

As in the case where $p$ is not equal to the characteristic of $\FF$, the element $L(A,\gamma^{-1})$ interpolates the values $L(A(\epsilon^r),1)$, where $A(\epsilon^r)$ denotes the twist of $A$ by the $r$-th power of the cyclotomic character $\epsilon$. However, different from this situation, we obtain no information on the values $L(A,q^n)$. In particular, we still lack a noncommutative analogue of the description of these values as obtained by Milne \cite{Mil:VZFVFF} in the commutative case.

\bibliographystyle{amsalpha}
\bibliography{Literature}
\end{document}